\newif\ifdraft

\drafttrue

\ifdraft

\documentclass[11pt, reqno]{amsart}
\usepackage{lmodern}

\usepackage[a4paper, margin=1in]{geometry}

\usepackage[inline]{showlabels}

\else

\documentclass[reqno]{amsart}
\usepackage{lmodern}

\usepackage[a4paper, margin=.75in]{geometry}
\fi

\usepackage{amsmath, amsthm, thmtools, amsfonts, amssymb, mathtools}
\usepackage{pdflscape, blkarray, multirow, booktabs}

\usepackage{amstext} 
\usepackage{array}   
\newcolumntype{L}{>{$}l<{$}} 

\usepackage[dvipsnames]{xcolor}
\usepackage{hyperref}
\hypersetup{
	colorlinks = true,
	linkcolor = {Blue},
	citecolor = {BrickRed},
}

\usepackage{makecell}

\newcommand{\R}{\textup{R}}
\newcommand{\delr}[1]{\Delta^{#1}\left(\textup{R}_8\right)}
\newcommand{\delrn}[2]{\Delta^{#1}\left(\textup{R}_{#2}\right)}
\newcommand{\Z}{\mathbb{Z}}
\newcommand{\e}[1]{e_{#1}}
\newcommand{\dsum}{\oplus}
\newcommand{\defeq}{\vcentcolon=}

\theoremstyle{definition}
\newtheorem{thm}{Theorem}[section]
\newtheorem{lemma}[thm]{Lemma}
\newtheorem*{conj}{Conjecture}
\newtheorem{propositionX}{Proposition}

\begin{document}
\allowdisplaybreaks
\title[Counterexample to conjecture]{Counterexample to a conjecture about dihedral quandle}

\author[S. Panja]{Saikat Panja}
\address{Department of Mathematics, IISER Pune \\ Maharashtra, India}
\email{panjasaikat300@gmail.com}
\author[S. Prasad]{Sachchidanand Prasad}
\address{Department of Mathematics and Statistics, IISER Kolkata \\ West Bengal, India}
\email{sp17rs038@iiserkol.ac.in}

\subjclass[2010]{}


\begin{abstract}
    It was conjectured that the augmentation ideal of a dihedral quandle of even order $n>2$ satisfies $|\Delta^k(\R_n)/\Delta^{k+1}(\R_{n})|=n$ for all $k\ge 2$. In this article we provide a counterexample against this conjecture.
\end{abstract}
\date{\today}

\maketitle

\setcounter{tocdepth}{3}

\frenchspacing 

\section{Introduction} \label{sec:introduction}
A \textit{quandle} is a pair $(A,\cdot)$ such that `$\cdot$' is a binary operation satisfying
\begin{enumerate}
	\item the map $S_a:A\longrightarrow A$, defined as $S_a(b)=b\cdot a$ is an automorphism for all $a\in A$,
	\item  for all $a\in A$, we have $S_a(a)=a$.
\end{enumerate}

\noindent To have a better understanding of the structure, a theory parallel to group rings was introduced by Bardakov, Passi and Singh in \cite{BaPaSi19}. Let $\Z_n$ denote the cyclic group of order $n$. Then defining  $a\cdot b=2b-a$ defines a quandle structure on $A=\Z_n$. This is known as \textit{dihedral quandle}. For other examples see \cite{BaPaSi19}. The quandle ring of a quandle $A$ is defined as follows. Let $R$ be a commutative ring. Consider 
\begin{displaymath}
	R[A] \defeq \left\{\sum_{i}r_ia_i: r_i\in R,a_i\in A  \right\}.
\end{displaymath}
Then this is an additive group in usual way. Define multiplication as 
\begin{displaymath}
	\left(\sum_{i}r_ia_i\right)\cdot \left(\sum_{j}s_ja_j\right) \defeq \sum_{i,j}r_is_j(a_i\cdot a_j).
\end{displaymath}
The \textit{augmentation ideal} of $R[A]$, $\Delta_R(A)$ is defined as the kernel of the augmentation map
\begin{displaymath}
	\varepsilon :R[A]\to R,~\sum_{i}r_ia_i \mapsto \sum_{i} r_i.
\end{displaymath}
The powers $\Delta^k_R(A)$ is defined as $\left(\Delta_R(A)\right)^k$. When $R=\Z$, we will be omitting the subscript $R$. The following proposition gives a basis for $\Delta_R(X)$.

\begin{propositionX}\cite[Proposition 3.2, Page 6]{BaPaSi19}  \label{prop:basis}
	A basis of $\Delta_R(X)$ as an $R$-module is given by $\{a-a_0:a\in A\setminus\{a_0\}\}$, where $a_0\in A$ is a fixed element. 
\end{propositionX}

 The following has been conjectured in \cite[Conjecture 6.5, Page 20] {BaPaSi19}.
\begin{conj}
	Let $\R_n=\{a_0,a_1,\cdots,a_{n-1}\}$ denote the dihedral quandle of order $n$. Then we have the following statements.
	\begin{enumerate}
		\item For an odd integer $n>1$, $\delrn{k}{n}/\delrn{k+1}{n}\cong \Z_n$ for all $k\ge 1$.
  		\item For an even integer $n> 2$, $\left|\delrn{k}{n}/\delrn{k+1}{n}\right|=n$ for $k\ge 2$.
	\end{enumerate}
	The first statement has been confirmed by Elhamdadi, 
	Fernando and Tsvelikhovskiy in \cite[Theorem 6.2, Page 182]{ElFeTs19}. The second statement holds true for $n=4$, see \cite{BaPaSi19}. Here we have given a counterexample in \autoref{thm:mainTheorem} to show that the conjecture is not true in general. 
	
\end{conj}

\section{Counterexample}\label{sec:counterexample}

\begin{thm} \label{thm:mainTheorem}
    Let $\R_8$ be the dihedral quandle of order $8$. Then 
    \begin{displaymath}
        \left|\Delta^2\left(\R_8\right)/\Delta^3\left(\R_8\right)\right|= 16.
    \end{displaymath}
\end{thm}

\noindent From \autoref{prop:basis}, we get that $\{e_i=a_i-a_0:i=1,2,\cdots, n-1\}$ is a basis for $\delrn{}{n}$. We will be using this notation in the subsequent computation.

\begin{lemma}\label{lemma:multiplictionWith_e4}
    Let $\R_{2k}$ denote the dihedral quandle of order $2k~(k\ge 2)$. Then $e_i \cdot e_k=0$ for all $i=1,2,\cdots, 2k-1$.
\end{lemma}

\begin{proof}
    Observe that 
    \begin{align*}
        e_i \cdot e_k & = \left(a_i -a_0\right) \cdot \left(a_k-a_0\right) \\
        & = a_{2k-i}-a_{2k-i}-a_0+a_0=0.
    \end{align*}
\end{proof}

\begin{lemma}\label{lemma:multiplictionSymmetry}
    Let $\R_{2k}$ denote the dihedral quandle of order $2k~(k\ge 2)$. Then $e_i\cdot e_j = e_i \cdot e_{k+j}$ for all $j=1,2,\cdots,k-1$ and for all $i=1,2,\cdots,2k-1$.
\end{lemma}
\begin{proof}
    Note that 
    \begin{align*}
        e_i \cdot e_{k+j} & = a_ia_{k+j}-a_ia_0-a_0a_{k+j}+a_0 \\
        & = a_i a_j - a_i a_0 -a_0a_j+a_0 \\
        & = e_i \cdot e_j.
    \end{align*}
\end{proof}
\noindent We will use \autoref{lemma:multiplictionWith_e4} and \autoref{lemma:multiplictionSymmetry} to simplify the multiplication tables. 
\begin{proof}[Proof of \autoref{thm:mainTheorem}]
    Recall that a basis of $\delr{}$ is given by $\mathcal{B}_1=\{e_1,e_2,\cdots,e_7\}$. The multiplication table for the $e_i\cdot e_j$ is given as follows:
    \begin{center}
        \begin{displaymath}
            \begin{array}{|c|c|c|c|}
            \hline
                     & e_1 & e_2 & e_3  \\ \hline 
                 e_1 & e_1-e_2-e_7 & e_3-e_4-e_7 & e_5-e_6-e_7 \\
                 \hline
                 e_2 & -e_2-e_6 & e_2-e_4-e_6 & -2e_6 \\
                 \hline 
                 e_3 & -e_2-e_5+e_7 & e_1-e_4-e_5& e_3-e_5-e_6 \\ 
                 \hline
                 e_4 &  -e_2-e_4+e_6 & -2e_4 & e_2 - e_4- e_6 \\ 
                 \hline
                 e_5 & -e_2-e_3+e_5 & -e_3-e_4+e_7 & e_1-e_3-e_6 \\
                 \hline
                 e_6 & -2e_2 + e_4 & -e_2 - e_4 + e_6 & -e_2-e_6 \\
                 \hline
                 e_7 & -e_1-e_2 + e_3 & -e_1-e_4+e_5 & -e_1-e_6+e_7 \\
                 \hline
            \end{array}
        \end{displaymath}
    \end{center}
    Since $\delr{2}$ is generated by $e_i\cdot e_j$ as a $\Z$-module, using row reduction over $\Z$ one can show that a $\Z$-basis is given by 

    \begin{align*}
        \mathcal{B}_2 = & \left\{u_1 = \e{1}-\e{2}-\e{7}, u_2 = \e{2}+\e{6}, u_3= \e{3}-\e{4}-\e{7},\right. \\ 
        & \kern .5cm \left.u_4 = \e{4}+2\e{6}, u_5 = \e{5}-\e{6}-\e{7}, u_6 = 4\e{6} \right\}.
    \end{align*}
    We now want to express a $\Z$-basis of $\delr{3}$ in terms of $\mathcal{B}_2$. First we calculate the products $u_i\cdot e_j$. This is presented in the following table.
    \begin{center}
        \begin{displaymath}
            \begin{array}{|c|c|c|c|}
            \hline
                     & e_1 & e_2 & e_3  \\ \hline 
                 u_1 & \makecell{2e_1 + e_2 -e_3 \\ +e_6 -e_7} & \makecell{e_1 -e_2 +e_3 \\+e_4 -e_5 +e_6 -e_7 }& \makecell{e_1 -e_4 +e_5 \\ +2e_6 -2e_7}  \\
                 \hline
                 u_2 & -3e_2+e_4 -e_6 & -2e_4 & -e_2 +e_4 -3e_6  \\
                 \hline 
                 u_3 & \makecell{e_1+e_2-e_3\\+e_4-e_5-e_6+e_7} & 2e_1+2e_4-2e_5& \makecell{e_1-e_2+e_3+e_4 \\-e_5 +e_6 -e_7}  \\ 
                 \hline
                 u_4 &  -5e_2-e_4+e_6 & -2e_2-4e_4+2e_6 & -e_2-e_4 -3e_6 \\ 
                 \hline
                 u_5 & \makecell{e_1+2e_2-2e_3\\-e_4+e_5} & \makecell{e_1+e_2-e_3+e_4\\-e_5-e_6+e_7} & 2e_1+e_2-e_3+e_6-e_7 \\
                 \hline
                 u_6 & -8e_2+4e_4 & -4e_2-4e_4+4e_6 & -4e_2-4e_6 \\
                 \hline
            \end{array}
        \end{displaymath}
    \end{center}
    
    \noindent Hence, a $\Z$-basis for $\delr{3}$ is given by
    \begin{align*}
        \mathcal{B}_3 & = \left\{v_1 = e_1-e_2+e_3+e_4-e_5+e_6-e_7, v_2 = e_2 - e_3 -2e_4+2e_5+e_6-e_7, \right. \\
        & \kern 0.5cm \left. v_3 = -e_3-e_4+2e_5-2e_6-e_7, v_4 = -2e_4, v_5 = -4e_5-4e_6 + 4e_7, v_6 = 8e_6 \right\}.
    \end{align*}
    Now we will present the elements of $\mathcal{B}_3$ in terms of $\mathcal{B}_2$. We have the following presentation. 
    \begin{displaymath}
        \begin{array}{c c c c c c c c}
            v_1 & = & u_1 & & & + 2u_4 & -u_5 & -u_6 \\
            v_2 & = &     & u_2 & -u_3 & - u_4 & + 2u_5 & + u_6 \\
            v_3 & = & & & -u_3 & -2u_4 & +2u_5 & +u_6 \\
            v_4 & = & & & & 2u_4 & & -u_6\\
            v_5 & = & & & & & -4u_5 \\
            v_6 & = & & & & & & 2u_6.
        \end{array}
    \end{displaymath}
    Note that we can alter the basis $\mathcal{B}_2$ of $\delr{2}$ as follows:
    \begin{align*}
        & \left\{u_1+2u_4-u_5-u_6, u_2-u_3-u_4+2u_5+u_6, u_3+2u_4-2u_5-u_6, u_4, u_5, u_6 \right\}.
    \end{align*}
    Hence,
    \begin{align*}
        \dfrac{\delr{2}}{\delr{3}} & \cong \dfrac{\Z v_1\dsum \Z v_2 \dsum \Z v_3 \dsum \Z u_4\dsum \Z u_5 \dsum \Z u_6}{\Z v_1\dsum \Z v_2 \dsum \Z v_3 \dsum \Z (2u_4-u_6)\dsum \Z (-4u_5) \dsum \Z (2u_6)} \\
        & \cong \Z_4\dsum \dfrac{\Z u_4 \dsum \Z u_6}{\Z (2u_4-u_6) \dsum \Z (2u_6)} \\ 
        & \cong \Z_4 \dsum \dfrac{\Z u_4 \dsum \Z u_6}{\Z u_4 \dsum \Z (4u_6)} \\
        & \cong \Z_4 \dsum \Z_4.
    \end{align*}
\end{proof}

\noindent\textbf{Acknowledgements:} The first author (Panja) acknowledges the support of NBHM PhD fellowship. The second author (Prasad) was supported by UGC (NET)-JRF fellowship. 

\bibliographystyle{alphaurl}

\end{document}